\documentclass[12pt]{amsart}
\usepackage{amsmath}
\usepackage{amsfonts,amssymb}
\usepackage{amsthm}
\usepackage[matrix,arrow,curve]{xy}
\numberwithin{equation}{section}
\theoremstyle{plain}
\newtheorem{theorem}{Theorem}[section]
\newtheorem{lemma}[theorem]{Lemma}
\newtheorem{predl}[theorem]{Proposition}
\newtheorem{corollary}[theorem]{Corollary}

\theoremstyle{definition}
\newtheorem{definition}[theorem]{Definition}
\newtheorem{remark}[theorem]{Remark}
\newtheorem{example}[theorem]{Example}

\newcommand{\quot }{/\!\!/}

\newcommand{\Z}{\mathbb Z}

\renewcommand{\AA}{\mathcal A}
\newcommand{\BB}{\mathcal B}
\newcommand{\CC}{\mathcal C}
\newcommand{\DD}{\mathcal D}
\renewcommand{\SS}{\mathcal S}
\newcommand{\RR}{\mathcal R}

\newcommand{\D}{\mathcal D}

\newcommand{\FF}{\mathcal F}

\newcommand{\LL}{\mathcal L}

\newcommand{\TTT}{\mathcal T}
\newcommand{\TT}{\mathsf T}
\renewcommand{\O}{\mathcal O}

\renewcommand{\k}{\mathsf k}
\newcommand{\Mod}{\mathrm{Mod}}
\newcommand{\mmod}{\mathrm{{-}mod}}

\newcommand{\coh}{\mathrm{coh}}

\newcommand{\Perf}{\mathrm{Perf}}
\newcommand{\Id}{\mathrm{Id}}
\newcommand{\id}{\mathrm{id}}

\newcommand{\pr}{\mathrm{pr}}
\newcommand{\Pretr}{\mathrm{Pre}\text{-}\mathrm{Tr}}
\newcommand{\vect}{\mathrm{vect}}

\newcommand{\ra}{\mathbin{\rightarrow}}

\newcommand{\xra}{\xrightarrow}
\renewcommand{\le}{\leqslant}
\renewcommand{\ge}{\geqslant}
\renewcommand{\~}{\widetilde}
\newcommand{\bul}{\bullet}

\def\a{\alpha}
\def\b{\beta}
\newcommand{\g}{\gamma}
\newcommand{\e}{\varepsilon}
\newcommand{\s}{\sigma}

\renewcommand{\th}{\theta}

\DeclareMathOperator{\Hom}{\textup{Hom}}

 \DeclareMathOperator{\Pic}{\mathrm{Pic}}

 \DeclareMathOperator{\Spec}{\mathrm{Spec}}
\DeclareMathOperator{\har}{\mathrm{char}} \DeclareMathOperator{\Ob}{\mathrm{Ob}}

\begin{document}

\author{Alexey Elagin}
\address{Institute for Information Transmission Problems (Kharkevich Institute), Bolshoy Karetny per. 19, Moscow,  Russia, 127994;}
\address{National Research University Higher School of Economics (HSE), 
AG Laboratory, Vavilova str. 7, Moscow, Russia, 117312.}
\email{alexelagin@rambler.ru} 
\title{On equivariant triangulated categories}
\thanks{The author was partially supported  
by AG Laboratory HSE (RF government grant, ag. 11.G34.31.0023), by the Presidents grant NSh-2998.2014.1, by RFBR grants 15-01-02158 and 15-51-50045 and
by the Dynasty foundation.}
\date{}
\begin{abstract}                                                      
Consider a finite group $G$ acting on a triangulated category $\TTT$. In this paper we investigate triangulated structure on the category $\TTT^G$ of $G$-equivariant objects in $\TTT$. We prove (under some technical conditions) that such structure exists. Supposed that an action on $\TTT$ is induced by a DG-action on some DG-enhancement of $\TTT$, we construct a DG-enhancement of $\TTT^G$. Also, we show that the relation ``to be an equivariant category with respect to a finite abelian group action'' is symmetric on idempotent complete additive categories.
\end{abstract}

\maketitle

\section{Introduction}
Triangulated categories became very popular in algebra, geometry and topology in last decades.
In algebraic geometry, they arise as derived categories of coherent sheaves on algebraic varieties or stacks. It turned out that some geometry of varieties can be understood well through their derived categories and homological algebra of these categories. 
Therefore it is always interesting and important to understand how different geometrical operations, constructions, relations look like on the derived category side.

In this paper we are interested in autoequivalences of derived categories or, more  general, in group actions on triangulated categories. 
For $X$ an algebraic variety, there are ``expected'' autoequivalences of $\D^b(\coh(X))$ which are induced by automorphisms of $X$ or by tensoring into line bundles on $X$.
If $X$ is a smooth Fano or if $K_X$ is ample, essentially that is all: 
Bondal and Orlov  have shown in~\cite{BO} that for smooth irreducible projective variety $X$ with $K_X$ or $-K_X$ ample all autoequivalences of $\D^b(\coh(X))$ are generated by automorphisms of $X$, twists into line bundles on $X$ and translations. On the contrary,  varieties with zero $K_X$ may have many non-trivial autoequivalences of $\D^b(\coh(X))$. 
For example, the autoequivalence group of derived category of abelian varieties was calculated by Orlov in~\cite{Or1}.

Our goal is to study, for an action of a group $G$ on a triangulated category $\TTT$, the ``quotient category'' $\TTT^G$, or the category of $G$-equivariant objects in $\TTT$.

The motivation comes from the concept of a $G$-equivariant sheaf. If $X$ is an algebraic variety and the group $G$ (finite or reductive algebraic) acts freely on $X$, then $G$-equivariant coherent sheaves on $X$ correspond to coherent sheaves on the quotient variety $X/G$. On the categorical level, the category $\coh^G(X)$ of $G$-equivariant coherent sheaves on $X$ is equivalent to the category $\coh(X/G)$. For arbitrary $G$-actions, $G$-equivariant sheaves correspond to sheaves on the quotient stack $X\quot G$
which is different from the quotient variety in general.

Following Deligne~\cite{De}, one can define action of a group on a category  and consider equivariant objects in the category with respect to the action, see Section~\ref{section_group}. For an action of a group $G$ on a category $\CC$, we denote the category of $G$-equivariant objects in~$\CC$ by $\CC^G$. For $\CC=\coh(X)$ and the action on $\CC$ induced by an action of $G$ on a variety $X$, $G$-equivariant objects in $\CC$ are $G$-equivariant sheaves on $X$. 
Another basic example of a group action on $\coh(X)$ comes from twisting into line bundles.
If $G\subset \Pic(X)$ is a finite subgroup in the Picard group of~$X$, then tensor product with bundles from $G$ defines an action of $G$ on $\coh(X)$. In this case, $G$-equivariant objects in $\coh(X)$ correspond to  coherent sheaves on a non-ramified $|G|$-fold cover of $X$ which is given explicitly as the relative spectrum $\Spec_X(\oplus_{\LL\in G}\LL^{-1})$. 

Instead of abelian categories of sheaves, one could consider derived categories and group actions on them. 
What categories would equivariant objects form? For the two examples of actions mentioned above the result is not surprising. We have

\begin{theorem}[first proved in \cite{El2}]
\label{th_descentforequiv_intro}
Suppose  $G$ is a group, $X$ is an algebraic $G$-variety over $\k$ and $\har(\k)$ does not divide $|G|$. Then $\D^b(\coh(X))^G\cong \D^b(\coh^G(X))$.
\end{theorem}

and 

\begin{theorem}[see Theorem \ref{th_descentforcoverings} below]
\label{th_descentforcoverings_intro}
Let $X$ be an algebraic variety over a field $\k$ and $G\subset \Pic(X)$ be a finite subgroup. Suppose $\har(\k)$ does not divide $|G|$. 
Let $G$ act on  $\D^b(\coh(X))$ by twisting into line bundles~$\LL\in G$.
Denote by $Y$ the relative spectrum 
$\Spec_X(\oplus_{\LL\in G}\LL^{-1}).$
Then $\D^b(\coh(X))^G\cong \D^b(\coh(Y))$.
\end{theorem}

Thus equivariant objects in triangulated categories are of some interest.
Let $\TTT$ be a triangulated category, suppose one has an exact action of a group $G$ on $\TTT$. Consider the ``quotient category'' $\TTT^G$. It is natural to ask whether $\TTT^G$ has any natural triangulation. Surprisingly, the answer is positive (up to some technical details), see Theorem~\ref{th_TTTGtriang}. We deduce it from results of P.\,Balmer \cite{Ba2}.

Usually when one gets a triangulated category it comes with some additional structure, like a DG-enhancement. The next question is: for an action of $G$ on $\TTT$, is it possible to construct a DG-enhancement of $\TTT^G$ given a DG-enhancement of $\TTT$?
We make the first step answering this question: construct an enhancement of $\TTT^G$ provided that the $G$-action on $\TTT$ is induced by a $G$-action on the enhancement.
More concretely, starting with a $G$-action on a pretriangulated DG-category $\AA$, we construct DG-category $Q_G(\AA)$ being a DG-enhancement of $H^0(\AA)^G$ (see Theorem~\ref{th_maingeneral}).
Similar construction was supposed by P.\,Sosna in paper~\cite{So}, which motivated our research of equivariant DG-categories. Compared to his one, our construction has better functorial properties, in particular, quasi-equivalent $\AA$ and 
$\AA'$ produce quasi-equivalent $Q_G(\AA)$ and $Q_G(\AA')$.

Another similar (and more simple) situation when the category $\TTT^G$ can be well-understood, is the following: $\TTT$ is the bounded derived category of some abelian category $\AA$ and the action on $\TTT$ is induced by an action on $\AA$. In this case $\TTT^G$ is just equivalent to the bounded derived category of $\AA^G$. This is not an original result: see for example X.-W.\,Chen \cite{Ch}, but we give the proof for the convenience of the reader.

The above case covers many interesting examples of group action on derived categories of coherent sheaves on algebraic varieties, like in Theorem \ref{th_descentforequiv_intro} or Theorem \ref{th_descentforcoverings_intro}.

In the general setting, the task of finding a DG-enhancement for the category $\TTT^G$ is not so easy. To use our above mentioned construction of $Q_G(\AA)$, it suffices to lift a group action on a triangulated category onto DG-level. In particular, it requires to lift an autoequivalence of a
triangulated category to a DG-autoequivalence of certain DG-enhancement. Even in a good geometric situation: for smooth projective algebraic variety~$X$ and an  autoequivalence of $\D^b(\coh(X))$ given by a kernel on $X\times X$, it is not clear how to do it. It is known (see Lunts and Orlov's paper~\cite{LO}) that for a projective variety $X$, DG-enhancement of $\D^b(\coh(X))$ is strongly unique.	But this uniqueness is too flexible, it allows to lift action onto DG-level only in a very weak sense: a sense of quasi-functors, which is not suitable for our purposes.
The problem of constructing a DG-enhancement of $\D^b(\coh(X))$ with a compatible group action on it seems to be rather interesting to investigate.

Our treatment is based on descent theory. This point of view was developed by the author in~\cite{El2} and involves the language of monads and modules.
We consider equivariant categories like $\TTT^G$ or $\D^b(\coh(X))^G$ as certain categories of descent data. Namely, with any action of a group $G$ on a category $\CC$ 
we associate a comonad $\TT_G$ on $\CC$ such that the corresponding category of comodules is equivalent to $\CC^G$. Thus, key point in the proof of Theorems~\ref{th_AGAG} and \ref{th_main} is to  show that a certain comparison functor is an equivalence. This is done using a rather specific but powerful special case of Beck theorem (see Proposition~\ref{prop_suffcond}) which is valid for triangulated categories.

In Section~\ref{section_reversion} one more application of our methods is given. We provide a nice categorical generalization of the following observation. 
Consider a Galois covering $X\to Y$ of algebraic varieties with an abelian Galois group. Theorems~\ref{th_descentforequiv_intro} and \ref{th_descentforcoverings_intro} imply that either of the categories $\D^b(\coh(X))$ and $\D^b(\coh(Y))$ can be reconstructed from another one as a category of equivariant objects with respect to a certain group action.

Using the language of monads, we demonstrate that this situation is typical, proving the following reversion theorem:
\begin{theorem}
Let $\BB$ and $\CC$ be idempotent complete additive categories over an algebraically closed  field $\k$, suppose  $\har(\k)$ does not divide $|G|$. Suppose $\BB\cong \CC^G$ for some action of a finite abelian group $G$ on $\CC$. Then $\CC\cong \BB^{G^{\vee}}$ for some action of the dual group $G^{\vee}$ on $\BB$.
\end{theorem}

The paper is organized as follows.
In Section ~\ref{section_monads} we recall necessary facts about monads and comonads. In Section ~\ref{section_group} we define group actions and equivariant objects and 
introduce adjoint functors, monads and comonads needed for descent theory.
In Section~\ref{section_reversion} we apply methods from Section~\ref{section_group} and prove ``reversion theorem'' for abelian group actions. This result is not needed for the sequel.
In Section~\ref{section_trcat} we recall basics about triangulated categories and DG-categories.
In Section~\ref{section_balmer}, for a triangulated category $\TTT$ with a group action, we define shift
functor and a class of distinguished triangles in $\TTT^G$. Using Balmer's results, we explain that $\TTT^G$ is also a triangulated category (up to some technical details).
In Sections~\ref{section_derived} and \ref{section_dg} we go further and show that the category $\TTT^G$ can be well-understood in the two important cases. In Section~\ref{section_derived} we consider action of $G$ on a derived category $\D^b(\AA)$ of an abelian category $\AA$, induced by a $G$-action on $\AA$ and demonstrate that $\D^b(\AA)^G\cong \D^b(\AA^G)$. In Section~\ref{section_dg} we consider action of $G$ on a DG-enhanced triangulated category $\TTT$, induced by a $G$-action on a DG-enhancement $\AA$. In this setting we construct a DG-category $Q_G(\AA)$ which is an enhancement of $\TTT^G$. 

I thank Alexander Kuznetsov, Dmitry Orlov and Evgeny Shinder for valuable remarks and suggestions, and Sergey Galkin for his motivating interest to equivariant DG-categories.

\section{Preliminaries on (co)monads}
\label{section_monads}

We recall some facts concerning (co)monads and (co)modules. More details can be found in books by
Barr-Wells~\cite[chapter 3]{TTT} and MacLane~\cite[chapter 6]{ML}.

\medskip
Let $\CC$ be a category.
\begin{definition}
\label{def_comonad} A~\emph{comonad} $\TT=(T,\e,\delta)$  on the category
$\CC$ consists of a functor $T\colon \CC\ra \CC$ and of natural transformations of functors $\e\colon T\ra
\Id_{\CC}$ and $\delta\colon T\ra T^2=TT$ such that the following diagrams are commutative:
$$\xymatrix{
T \ar[r]^{\delta} \ar@{=}[rd] \ar[d]^{\delta} & T^2 \ar[d]^{T\e} \\ T^2 \ar[r]^{\e
T} & T, }\qquad \xymatrix{T \ar[r]^{\delta} \ar[d]^{\delta} & T^2 \ar[d]^{T\delta}\\
T^2 \ar[r]^{\delta T} & T^3. }
$$
\end{definition}

\begin{definition}
Two comonads
$\TT=(T,\e,\delta)$  and $\TT'=(T',\e',\delta')$  on the category $\CC$ are \emph{isomorphic} if there exists an isomorphism of functors $T\to T'$ compatible with $\e$-s and $\delta$-s.
\end{definition}

\begin{example}
\label{mainexample} Consider a pair of adjoint functors: $P^* \colon\BB\ra \CC$
(left) and $P_*\colon \CC\ra \BB$ (right). Let $\eta\colon \Id_{\BB}\ra P_*P^*$ and
$\e\colon P^*P_*\ra \Id_{\CC}$ be the natural adjunction morphisms. Define a triple $(T,\e,\delta)$ by taking $T=P^*P_*$ and $\delta=P^*\eta P_* \colon P^*P_*\ra P^*P_*P^*P_*$. Then $\TT=(T,\e,\delta)$ is a comonad on $\CC$.
\end{example}

\begin{definition}
\label{def_tpp}
The comonad introduced above will be denoted $\TT(P^*,P_*)$.
\end{definition}

Essentially, any comonad can be obtained in this way from a pair of adjoint functors. This follows from the below construction due to Eilenberg-Moore.

\begin{definition}
\label{def_comodule}
Suppose $\TT=(T,\e,\delta)$ is a comonad on $\CC$.
A \emph{comodule} over $\TT$ (it is sometimes called a \emph{$\TT$-coalgebra}) is a pair
$(F,h)$ where $F\in \Ob  \CC$ and $h\colon F\ra TF$ is a morphism
satisfying the following two conditions:
\begin{enumerate}
\item the composition
$$F\xra{h} TF \xra{\e F} F$$
is the identity;
\item the diagram
$$\xymatrix{
F \ar[r]^h \ar[d]^h  &  TF \ar[d]^{Th} \\
TF \ar[r]^{\delta F} & T^2F }$$ commutes.
\end{enumerate}
A \emph{morphism} between
comodules is defined in the natural way.

\end{definition}

All comodules over a given comonad $\TT$ on $\CC$ form a category
which is denoted $\CC_{\TT}$. Define a functor $Q_*\colon \CC\ra
\CC_{\TT}$ by
$$Q_*F=(TF,\delta F),\qquad Q_*f=Tf,$$
define $Q^*\colon \CC_{\TT}\ra \CC$ to be the forgetful functor: $(F,h)\mapsto F$. Then
the pair of functors $(Q^*,Q_*)$ is an adjoint pair and 
the comonad $\TT(Q^*,Q_*)$ (see Definition~\ref{def_tpp}) is $\TT$.

\begin{predl}[Comparison theorem, {\cite[3.2.3]{TTT}, \cite[6.3]{ML}}]
\label{th_comparison} Assume that a comonad $\TT=(T,\e,\delta)$ on
a category $\CC$ is defined by an adjoint pair of functors
$P^*\colon \BB\ra \CC, P_*\colon \CC\ra \BB$. Then there exist a
unique (up to an isomorphism) functor (called \emph{comparison
functor}) $\Phi\colon \BB\ra \CC_{\TT}$ such that the diagram of
categories
$$\xymatrix{
&& {\BB} \ar[dd]^{\Phi} \ar@<1mm>[lld]^{P^*}\\
{\CC} \ar@<1mm>[rru]^{P_*} \ar@<-1mm>[rrd]_{Q_*} && \\
&&{\CC_{\TT}} \ar@<-1mm>[llu]_{Q^*} }$$ commutes, i.e. both triangles are commutative:
$$\Phi P_*\cong Q_*,\qquad Q^*\Phi\cong P^*.$$
\end{predl}

We want to understand when comparison functor
is an equivalence. Exact criterion is given by Beck theorem (see~\cite[3.14]{TTT} or \cite[6.7]{ML}) and is rather complicated. 
Below we present a simple sufficient condition on an adjoint pair providing comparison functor is an equivalence.

First we recall
\begin{definition}
\label{def_ic}
A category $\CC$ is called \emph{idempotent complete} (or \emph{Karoubian complete} or \emph{Cauchy complete}) if any projector in $\CC$ splits. That is, for any morphism 
$p\colon X\to X$ in $\CC$ such that $p^2=p$ there exists an object $X'$ in $\CC$  and a diagram in $\CC$
$$\xymatrix{
X'\ar@/^1em/[rr]^i && X\ar@/^1em/[ll]^s}$$
such that $si=1_{X'}, is=p$. The object $X'$ is called a \emph{retract} of $X$.
\end{definition}

\begin{predl}[see \cite{Me}, Corollary 3.17 and Proposition 3.18, or~\cite{El2}, Corollaries 3.10 and 3.11]
\label{prop_suffcond} 
In the above notation suppose that the category  $\BB$ is
idempotent complete. If the natural morphism of functors $\eta \colon
\Id_{\BB}\to P_*P^*$ is a split monomorphism (i.e. has a left inverse morphism of functors) then the comparison functor $\Phi\colon 
\BB\to \CC_{\TT}$ is an
equivalence.

If the natural morphism $\eta(F) \colon
F\to P_*P^*(F)$ splits for any object $F\in \BB$  then the comparison functor $\Phi\colon 
\BB\to \CC_{\TT}$ is fully faithful.

\end{predl}

The notion of a monad is dual to the notion of a comonad. 
We present below related definitions and facts.

\begin{definition}
\label{def_monad} A~\emph{monad} $\SS=(S,\eta,\mu)$ on a category 
$\CC$ consists of a functor $S\colon \CC\ra \CC$ and of natural transformations of functors $\eta\colon \Id_{\CC}\ra S$ and $\mu\colon T^2=TT\ra T$ such that the following diagrams are commutative:
$$\xymatrix{
S \ar[r]^{\eta S} \ar@{=}[rd] \ar[d]^{S\eta} & S^2 \ar[d]^{\mu} \\ S^2 \ar[r]^{\mu} & S, }\qquad 
\xymatrix{S^3 \ar[r]^{S\mu} \ar[d]^{\mu S} & S^2 \ar[d]^{\mu}\\
S^2 \ar[r]^{\mu} & S. }
$$
\end{definition}


\begin{definition}
\label{def_spp}
Consider a pair of adjoint functors: $P^* \colon\CC\ra \BB$
(left) and $P_*\colon \BB\ra \CC$ (right). The endofunctor $S=P_*P^*\colon \CC\to\CC$ together with natural adjunction morphisms forms a monad $\SS=(S,\eta,\mu)$ on $\CC$.
\end{definition}

\begin{definition}
\label{def_module}
Suppose $\SS=(S,\eta,\mu)$ is a monad on $\CC$.
A \emph{module} over $\SS$ is a pair
$(F,h)$ where $F\in \Ob  \CC$ and $h\colon SF\ra F$ is a morphism
satisfying the following two conditions:
$1_F=h\circ \eta F\colon F\to F$, 
$h\circ \mu F=h\circ Sh\colon S^2F\to F$.
\end{definition}

All modules over a given monad $\SS$ on $\CC$ form a category
which is denoted $\CC^{\SS}$. 
Define a functor $Q^*\colon \CC\ra
\CC^{\SS}$ by
$$Q^*F=(SF,\mu F),\qquad Q^*f=Sf,$$
let $Q_*\colon \CC^{\SS}\ra \CC$ be the forgetful functor. Then
the pair of functors $(Q^*,Q_*)$ is an adjoint pair and the monad $\SS(Q^*,Q_*)$  is $\SS$.

\begin{predl}[Comparison theorem for monads]
Assume that a monad $\SS=\SS(P^*,P_*)$ on
a category $\CC$ is defined by an adjoint pair of functors
$P^*\colon \CC\ra \BB, P_*\colon \BB\ra \CC$. Then there exists a
unique (up to an isomorphism) functor (called \emph{comparison
functor}) $\Phi\colon \BB\ra \CC^{\SS}$ such that the diagram of
categories
$$\xymatrix{
&& {\BB} \ar[dd]^{\Phi} \ar@<1mm>[lld]^{P_*}\\
{\CC} \ar@<1mm>[rru]^{P^*} \ar@<-1mm>[rrd]_{Q^*} && \\
&&{\CC^{\SS}} \ar@<-1mm>[llu]_{Q_*} }$$ commutes, i.e. both triangles are commutative:
$$\Phi P^*\cong Q^*,\qquad Q_*\Phi\cong P_*.$$
\end{predl}

Here we recall about idempotent completion of categories. We refer to~\cite{Ba} or~\cite{BD} for details. 

For any category $\CC$, there exists a fully faithful embedding $\CC\to\bar\CC$ to an idempotent complete category $\bar\CC$ such that any object in $\bar\CC$ is a retract of some object in $\CC$. Such embedding is unique up to an isomorphism, and category $\bar\CC$ is called \emph{idempotent completion} of $\CC$. We will need the following property of idempotent completion:

\begin{predl}[{\cite[Prop 5.1.4.9]{Lu}} or {\cite[Prop. 1.3]{Ba}}]
\label{predl_property}
For an idempotent completion $\bar\CC$ of a category $\CC$ and for an idempotent complete category $\DD$ one has an equivalence 
$$Fun(\bar\CC,\DD)\cong Fun(\CC,\DD),$$
where $Fun$ denotes the category of functors and natural transformations.
\end{predl}

\begin{predl}
\label{predl_extend1}
Let $\CC$ be a category and $\bar\CC$  be an idempotent completion of $\CC$. 
Let $\SS=(S,\eta,\mu)$ be a monad on $\CC$. Then~$\SS$ extends uniquely to a monad on $\bar\CC$.
\end{predl}
\begin{proof}
The functor $S\colon \CC\to\CC$ extends to a functor $\CC\to \bar\CC$ by embedding $\CC\to\bar\CC$. By Proposition~\ref{predl_property}, the latter functor corresponds to a functor $\bar\CC\to\bar\CC$ which we denote by $\bar S$. Further, a morphism of functors $\mu\colon S^2\to S\colon \CC\to\CC$ gives a morphism of their extensions to functors $\CC\to\bar\CC$. Again, by Proposition~\ref{predl_property} this morphism corresponds to a morphism $\bar\mu\colon \bar S^2\to\bar S$. Likewise we get a morphism $\bar\eta$. Clearly, we obtain a monad $(\bar S,\bar \eta,\bar \mu)$ on $\bar\CC$.
\end{proof}

\section{Preliminaries on group actions and equivariant objects}
\label{section_group}

Let $\CC$ be a pre-additive category, linear over a ring $\k$. Let $G$ be a finite group, suppose that $|G|$ is invertible in $\k$.

\begin{definition}
A \emph{(right) action} of $G$ on $\CC$ is defined by the following data:
\begin{itemize}
\item family of autoequivalences $\phi_g\colon \CC\to \CC, g\in G$;
\item family of isomorphisms $\e_{g,h}\colon \phi_g\phi_h\to \phi_{hg}$, 
for which all diagrams 
$$\xymatrix{\phi_f\phi_g\phi_h \ar[rr]^{\e_{g,h}} \ar[d]^{\e_{f,g}} && \phi_f\phi_{hg}\ar[d]^{\e_{f,gh}}\\
\phi_{gf}\phi_h\ar[rr]^{\e_{gf,h}} && \phi_{hgf} 
}$$
are commutative.
\end{itemize}
\end{definition}

\begin{remark}
We do not require $\phi_e$ to be the identity functor, but the definition implies they are naturally isomorphic. Indeed, we have an isomorphism of functors $\e_{e,e}\colon \phi_e\phi_e\to \phi_{e}$, and since $\phi_e$ is fully faithful, we get an isomorphism 
$$\phi_e^{-1}(\e_{e,e})\colon \phi_e\to\Id.$$  
Denote the inverse isomorphism $\Id\to \phi_e$ by $u$.
\end{remark}

\begin{remark}
It follows from cocycle condition for $\e$ that 
$$\phi_{e,g}=u^{-1}(\phi_g)\colon \phi_e\phi_g\to \phi_g, \quad 
\phi_{g,e}=\phi_g(u^{-1})\colon \phi_g\phi_e\to \phi_g.$$
In particular,
$$\e_{e,e}=\phi_e(u^{-1})=u^{-1}(\phi_e).$$
\end{remark}

\begin{example}
\label{example_action}
Suppose a group $G$ acts on a scheme $X$. Then $\phi_g=g^*\colon \coh(X)\to \coh(X)$
and canonical isomorphisms $g^*h^*\to (hg)^*$ define an action of $G$ on the category $\coh(X)$.
\end{example}

Suppose $G$ acts on a category $\CC$.
\begin{definition}
\label{def_C^G}
A \emph{$G$-equivariant} object in $\CC$ is a pair $(F,(\th_g)_{g\in G})$ where $F\in \Ob\CC$ and $(\th_g)_{g\in G}$ is a family of isomorphisms 
$$\th_g\colon F\to \phi_g(F),$$
such that all diagrams 
$$\xymatrix{
F\ar[rr]^{\th_g} \ar[d]^{\th_{hg}}&& \phi_g(F)\ar[d]^{\phi_g(\th_h)} \\
\phi_{hg}(F) && \phi_g(\phi_h(F)).\ar[ll]_{\e_{g,h}}
}$$
are commutative.
A \emph{morphism} of $G$-equivariant objects from $(F_1,(\th^1_g))$ to $(F_2,(\th^2_g))$ is a morphism $f\colon F_1\ra F_2$ compatible with $\theta_g$, i.e. such $f$ that
the below diagrams commute for all $g\in G$
$$\xymatrix{
F_1\ar[rr]^{\th^1_g} \ar[d]^f&& \phi_g(F_1)\ar[d]^{\phi_g(f)} \\
F_2\ar[rr]^{\th^2_g} && \phi_g(F_2).
}$$
The category of $G$-equivariant objects in $\CC$ is denoted $\CC^G$.
\end{definition}

\begin{remark}
It follows from the definition that ``$\th_e$ is identity''. More precisely, $\th_e\colon F\to \phi_e(F)$ equals $u(F)$.
\end{remark}

\begin{example}
In Example~\ref{example_action}, $G$-equivariant objects are $G$-equivariant coherent sheaves on $X$. 
\end{example}

Define the functor $p^*\colon \CC^G\to \CC$ as the forgetful functor: $p^*(F,(\th_g))=F$.

Define  the functor $p_*\colon \CC\to \CC^G$ as follows:
$$p_*(F)=\left(\bigoplus_{h\in G} \phi_h(F),(\xi_g)\right),$$
where 
$$\xi_g\colon \oplus_{h\in G} \phi_h(F)\to \oplus_{h\in G}\phi_g\phi_h(F)$$
is the collection of isomorphisms
$$\e_{g,h}^{-1}\colon \phi_{hg}(F)\to \phi_g\phi_h(F).$$

\begin{lemma}
\label{lemma_biadjunction}
The functor $p^*$ is both left and right adjoint to the functor $p_*$.
\end{lemma}
\begin{proof}
First check that $p^*$ is left adjoint to $p_*$. 
Construct the unit morphism
$\eta\colon \Id\to p_*p^*$ of endofunctors on $\CC^G$: for 
$\FF=(F,(\th_g))\in \CC^G$ take
$$\eta(\FF)=\sum_h \th_h\colon \FF\to p_*p^*\FF=(\oplus_h \phi_h(F),(\xi_g)).$$
Construct the counit morphism
$\e\colon p^*p_*\to\Id$ of endofunctors on $\CC$: for 
$F\in \CC$ take
$$\e(F)=u^{-1}(F)\pr_e\colon \oplus \phi_h(F)\to \phi_e(F)\to F.$$
One can check that these two morphisms	 satisfy necessary relations and hence the functors $p^*$ and $p_*$ are adjoint.

Likewise, to check that $p_*$ is left adjoint to $p^*$, we construct adjunction morphisms. 
Construct the unit 
$\eta'\colon \Id_{\CC}\to p^*p_*$: for 
$F\in \CC$ take
$$\eta'(F)\colon F\to p^*p_*F=\oplus_h \phi_h(F)$$ to be the composition of $u(F)\colon F\to \phi_e(F)$ and the embedding of the summand $\phi_e(F)$.

Construct the counit 
$\e'\colon p_*p^*\to\Id_{\CC^G}$: for 
$\FF=(F,(\th_g))\in \CC^G$ take
$$\e'(\FF)=\oplus_h \th_h^{-1}\colon p_*p^*\FF=(\oplus_h \phi_h(F),(\xi_g))\to \FF.$$
These two morphisms of functors satisfy certain relations, therefore $p_*$ is left adjoint to~$p^*$.
\end{proof}

Following Definitions~\ref{def_tpp} and \ref{def_spp}, one may consider 
\begin{itemize}
\item the comonad $\TT(p^*,p_*)$ on $\CC$;
\item the monad $\SS(p^*,p_*)$ on $\CC^G$;
\item the comonad $\TT(p_*,p^*)$ on $\CC^G$;
\item the monad $\SS(p_*,p^*)$ on $\CC$.
\end{itemize}

\begin{lemma}
\label{lemma_1G}
We have two equalities of natural transformations:
\begin{equation}
\e\circ \eta'=1_{\Id_{\CC}}
\end{equation}
and
\begin{equation}
\e'\circ \eta=|G|\cdot 1_{\Id_{\CC^G}}
\end{equation}
\end{lemma}
\begin{proof}
It follows immediately from explicit formulas, see Proof of Lemma~\ref{lemma_biadjunction}.
\end{proof}

\begin{definition}
\label{def_asscomonad}
The monad $\SS(p_*,p^*)$ and the comonad $\TT(p^*,p_*)$ on $\CC$ will be called \emph{associated with the action} of $G$ on $\CC$. 
\end{definition}
Next proposition shows that modules/comodules over these monad/comonad are precisely $G$-equivariant objects in $\CC$.

\begin{predl}
\label{prop_four}
The comparison functors 
\begin{enumerate}
\item $\CC^G\to \CC_{\TT(p^*,p_*)}$;
\item $\CC^G\to \CC^{\SS(p_*,p^*)}$

are equivalences. 

If, in addition, $\CC$ is idempotent complete, then
the comparison functors
\item $\CC\to (\CC^G)_{\TT(p_*,p^*)}$;
\item $\CC\to (\CC^G)^{\SS(p^*,p_*)}$

also are equivalences.
\end{enumerate}
\end{predl}

\begin{proof}

\begin{enumerate}
\item 
First, we prove that the comparison functor 
$\Phi\colon \CC^G\to \CC_{\TT(p^*,p_*)}$ is fully faithful.
According to~\ref{prop_suffcond}, we need to check that the unit of adjunction $\eta\colon \Id\to p_*p^*$ is a split embedding (for any object, but in fact the below splitting is functorial). 

Indeed, for any object $\FF\in \CC^G$ the morphism 
$\eta(\FF)\colon \FF\to p_*p^*\FF$
has a left inverse morphism $\frac1{|G|}\e'(\FF)$, see Lemma~\ref{lemma_1G}.

Then we check that $\Phi$ is essentially surjective. Indeed, take an object $(F,h)$ in
$\CC_{\TT(p^*,p_*)}$. Here $h\colon F\to TF=\oplus_h \phi_h(F)$ is a morphism obeying 
\begin{enumerate}
\item $\e (F)\circ h=\id_F$ and 
\item $T h\circ h=\delta (F)\circ h$.
\end{enumerate} 
Components of $h$ are some morphisms $\th_h\colon F\to \phi_h(F)$. Condition (a) imply that 
$\th_e$ is an isomorphism. Condition (b) imply that $\e_{h,g}\phi_h(\th_g)\th_h=\th_{gh}$, and therefore all $\th_h$ are invertible. We get that $(F,(\th_g))$ is an equivariant object and $(F,h)$ is isomorphic to $\Phi((F,(\th_g)))$.

\item is similar to (1).
\item Since $\CC$ is idempotent complete, one can use Proposition~\ref{prop_suffcond}. It suffices to verify that the unit of adjunction 
$\eta'\colon \Id\to p^*p_*$ has a left inverse morphism of functors.
Indeed, $\e$ is left inverse to $\eta'$ by Lemma~\ref{lemma_1G}.
\item is similar to (3).
\end{enumerate}
\end{proof}

\begin{example}
If $\CC$ is not idempotent complete, then comparison functors $(3)$ and $(4)$ from Proposition~\ref{prop_four} need not be equivalences. For example, take the category $\k-\vect$ of finite-dimensional vector spaces over a field $\k$ as $\CC$ and let $\CC_0\subset \CC$ be its subcategory of even-dimensional spaces. Let the group $G=\Z/2\Z$ act trivially on $\CC$. We claim that the comparisom functor 
$$\Phi_0\colon \CC_0\to (\CC_0^G)_{\TT(p^*,p_*)}$$
is not an equivalence.
To see this, consider the commutative diagram of functors
$$\xymatrix{
\CC\ar[rr]^{\Phi}&& (\CC^G)_{\TT(p^*,p_*)} \\
\CC_0\ar[rr]^{\Phi_0}\ar@{^{(}->}[u]^{\s}&& (\CC_0^G)_{\TT(p^*,p_*)}, \ar@{^{(}->}[u]^{\~\s}
}$$
where $\Phi$ and $\Phi_0$ are comparison functors and $\s, \~\s$ denote fully faithful embeddings. Category $\CC$ is idempotent complete, hence by Proposition~\ref{prop_four} $\Phi$ is an equivalence. For $V$ a vector space, $\Phi(V)$ is an object $(\mathcal V,h)$, 
where $\mathcal V=(p_*V,(\xi_g))=(V\oplus V, (\xi_g))\in \CC^G$. We see that $V\oplus V$ is even-dimensional and therefore $\Phi(V)$ belongs to the image of $\~\s$. It follows that $\~\s$ is an equivalence. Since $\s$ is not an equivalence, neither is $\Phi_0$.
\end{example}

We finish this section with
\begin{predl}
\label{predl_extend2}
Let $\CC$ be a category and $\bar\CC$  be an idempotent completion of $\CC$. 
Suppose a group $G$ acts on $\CC$. Then the action on $\CC$ extends uniquely to a $G$-action  on $\bar\CC$.
\end{predl}
The proof is analogous to the proof of Proposition~\ref{predl_extend1}, we skip it.

\section{Reversion for descent categories}
\label{section_reversion}

In this section we apply techniques from Section~\ref{section_group} to get a nice result: the relation ``to be a quotient category modulo finite abelian group action'' on the class of idempotent complete additive categories is symmetric. This result is not needed in the subsequent sections.

Assume that $\CC$ is an additive category and $G$ is a finite group acting on $\CC$. 
Define a comonad on the category $\CC^G$. 
Let $\k[G]$ be the regular representation of $G$ with the basis $e_g, g\in G$.
Take $R\colon \CC^G\to \CC^G$ to be the tensoring by the regular representation:
$$R((F,(\th_g)))=\k[G]\otimes (F,(\th_g)).$$
Define a morphism of functors  $\e_R\colon R\to \Id$ via the morphism of representations
$\k[G]\to \k$ such that $e_g\mapsto 1$. Define a morphism of functors $\delta_R\colon R\to RR$ via the morphism of representations
$\k[G]\to \k[G]\otimes\k[G]$ such that $e_g\mapsto e_g\otimes e_g$.
Clearly, we obtain a comonad $(R,\e_R,\delta_R)$, denote it by~$\RR$. 
\begin{predl}
\label{prop_TTRR}
Comonad  $\TT(p_*,p^*)$  is isomorphic to $\RR$.
\end{predl}
\begin{proof}
Define an isomorphism of endofunctors $\b\colon p_*p^*\to \k[G]\otimes -$.
For an object $\FF=(F,(\th_g))$ we take an isomorphism 
$$\b(\FF)\colon p_*p^*\FF = (\oplus_h \phi_h(F),(\xi_g))\to \k[G]\otimes \FF$$
which on the summand $\phi_h(F)$ is 
$$\th_h^{-1}\colon \phi_h(F)\to e_{h^{-1}}\otimes F.$$
Let us check that $\b(\FF)$ defines a morphism in $\CC^G$.
It follows from the diagram
$$\xymatrix{
\phi_h(F) \ar[rr]^{\th_h^{-1}} \ar[d]^{\xi_g} && e_{h^{-1}}\otimes  F\ar[d]^{\th_g}\\
\phi_g\phi_{hg^{-1}}(F) \ar[rr]^{\phi_g(\th_{hg^{-1}}^{-1})} && e_{gh^{-1}}\otimes \phi_g(F),
}$$
which is commutative by the definition of an equivariant object.

It remains to check that $\b$ is compatible with $\e$-s and  $\delta$-s, we leave it to the reader.
\end{proof}

From now on we suppose that the group $G$ is abelian and $\k$ is an algebraically closed field.
Let $G^{\vee}=\Hom(G,\k^*)$ be the dual group to $G$, that is, the group of characters of $G$. Define an action of $G^{\vee}$ on the category $\CC^G$ by twisting: for $\chi\in G^{\vee}$ let
$$\phi_{\chi}((F,(\th_h)))=(F,(\th_h))\otimes \chi=(F,(\th_h\cdot \chi(h))).$$
For $\chi,\psi\in G^{\vee}$ the equivariant objects $\phi_{\chi}(\phi_{\psi}((F,(\th_h))))$
and $\phi_{\psi\chi}((F,(\th_h)))$ are the same, let isomorphisms
$$\e_{\chi,\psi}\colon \phi_{\chi}\circ\phi_{\psi}\to \phi_{\psi\chi}$$
be identities.

\begin{theorem}
\label{th_isogeny}
Let $\k$ be an algebraically closed field, $G$ be a finite abelian group such that $\har(\k)$ does not divide $|G|$. Suppose $\CC$ is a $\k$-linear additive idempotent complete category and $G$ acts on $\CC$.

Then
$$(\CC^G)^{G^{\vee}}\cong \CC.$$
\end{theorem}
\begin{proof}
Consider the adjoint functors $q^*\colon (\CC^G)^{G^{\vee}}\to \CC^G$ and
$q_*\colon \CC^G\to (\CC^G)^{G^{\vee}}$ (see Section~\ref{section_group} for the definition). We claim that the comonad $\TT(q^*,q_*)$ on $\CC^G$
is isomorphic to the comonad $\RR$.

Indeed, the endofunctor $q^*q_*\colon \CC^G\to\CC^G$ is isomorphic to $R=\k[G]\otimes -$:
$$q^*q_*(\FF)=\oplus_{\chi\in G^{\vee}} (\chi\otimes \FF)\cong \k[G]\otimes \FF$$
since $\oplus_{\chi\in G^{\vee}} \chi\cong \k[G]$ as  a representation.
Fix an isomorphism $\g\colon \oplus_{\chi\in G^{\vee}} \chi\to \k[G]$ such that $\g(\oplus 1)=e_e$, we also denote by $\g$ the corresponding isomorphism $q^*q_*\to R$.

To check that $\g$ is compatible with counit and comultiplication in $\TT(q^*,q_*)$ 
and $\RR$, 
consider the diagrams of representations 
$$\xymatrix{\oplus_{\chi\in G^{\vee}}\chi \ar[rr]^\g \ar[d]^{\pr_{\chi_0}} && \k[G]\ar[d]^{\e_R}\\
\k\ar@{=}[rr] && \k,
}
\qquad
\xymatrix{\oplus_{\chi\in G^{\vee}}\chi \ar[rr]^\g \ar[d]^{\eta} && \k[G]\ar[d]^{\delta_R}\\
(\oplus_{\chi\in G^{\vee}}\chi)\otimes (\oplus_{\chi\in G^{\vee}}\chi ) \ar[rr]^{\g\otimes \g} && \k[G]\otimes \k[G]
}
$$
(here $\eta$ on the summand $\chi$ is a direct sum of identity maps to $\oplus_{\chi_1\chi_2=\chi} \chi_1\otimes \chi_2$).
It suffices to prove that both diagrams are commutative on the element $\oplus 1\in \oplus\chi$, which is true.

Therefore, one has:
$$(\CC^G)^{G^{\vee}}\cong (\CC^G)_{\TT(q^*,q_*)}\cong (\CC^G)_{\RR}\cong (\CC^G)_{\TT(p_*,p^*)}\cong \CC,$$
where the first and the fourth equivalences are due to Proposition~{\ref{prop_four}}, and the third one is by Proposition~\ref{prop_TTRR}.
\end{proof}

As an immediate corollary we get

\begin{theorem}
Let $\BB$ and $\CC$ be idempotent complete additive categories over an algebraically closed  field $\k$, suppose  $\har(\k)$ does not divide $|G|$. Suppose $\BB\cong \CC^G$ for some action of a finite abelian group $G$ on $\CC$. Then $\CC\cong \BB^{G^{\vee}}$ for some action of $G^{\vee}$ on $\BB$.
\end{theorem}


\begin{example}
Suppose $G$ is a finite abelian group.
Let $X$ be an algebraic $G$-variety over~$\k$, let $\CC=\D^b(\coh(X))$. Then
$\CC^G\cong \D^b(\coh^G(X))$ (see Theorem~\ref{th_descentforequiv_intro}), the group  $G^{\vee}$ acts on $\D^b(\coh^G(X))$ by twisting into characters of $G$. By Theorem~\ref{th_isogeny}, we have $$\D^b(\coh^G(X))^{G^{\vee}}\cong \D^b(\coh(X)).$$
\end{example}

The following special case is more geometric.

\begin{example}
Suppose $X$ is a Galois covering of an algebraic variety $Y$ with the abelian Galois group $G$. Then one has $\coh(Y)\cong \coh^G(X)$. By Theorem~\ref{th_descentforequiv_intro}, the same holds for derived categories:
$$\D^b(\coh(Y))\cong\D^b(\coh^G(X))\cong \D^b(\coh(X))^G.$$	
By Theorem~\ref{th_isogeny}, we get that 
\begin{equation}
\label{eq_reversion}
\D^b(\coh(X))\cong \D^b(\coh(Y))^{G^{\vee}}
\end{equation}
for some action of $G^{\vee}$ on $\D^b(\coh(Y))$.
Note that this action can be described explicitly. Indeed, under the equivalence $\coh^G(X)\cong \coh(Y)$ equivariant line bundles $\O_X\otimes \chi$, $\chi\in G^{\vee}$, on~$X$
correspond to some line bundles $\LL_{\chi}$ on~$Y$. The group  $G^{\vee}$ acts on $\D^b(\coh(Y))$ by tensoring in $\LL_{\chi}$. 
One can show that the relative spectrum $\Spec_Y(\oplus_{\chi\in G^{\vee}}\LL_{\chi}^{-1})$ is isomorphic to $X$. Therefore 
equivalence (\ref{eq_reversion}) also follows from Theorem~\ref{th_descentforcoverings_intro}.
\end{example}

\section{Preliminaries on triangulated and DG-categories}
\label{section_trcat}

In this section we collect definitions related to suspended and triangulated categories. Some of them are standard and can be found anywhere, others (concerning compatibility of functors and natural transformations with suspension) are more specific. We refer to a paper by P.\,Balmer, M.\,Schlichting~\cite{Ba}. See also Neeman's book~\cite{Ne}.
We expand here the definition of a triangulated category in order to make the further discussion of higher triangulated categories more consistent. Also we recall necessary facts about DG-categories and their homotopy categories.

We will adopt the following 
\begin{definition}
A \emph{suspended} category is an additive category $\CC$ with an automorphism functor (called \emph{suspension} or \emph{shift})
\end{definition}
We will denote shift functor by $X\mapsto X[1]\colon \CC\to\CC$. 

\begin{definition}
A functor $F$ between suspended categories $\CC_1$ and $\CC_2$ is called \emph{stable} if
it commutes with shifts: i.e., functors $F(-)[1]$ and $F(-[1])$ are equal.

A morphism of stable functors $\varphi\colon F\to H$ between suspended categories $\CC_1$ and $\CC_2$ is called \emph{stable} if  
it commutes with shifts: i.e., for any object $X\in\CC_1$ the diagram
$$\xymatrix{F(X)[1]\ar[rr]^{\varphi_X[1]}\ar@{=}[d] && H(X)[1]\ar@{=}[d]\\
F(X[1])\ar[rr]^{\varphi_{X[1]}} && H(X[1])
}$$
commutes.
\end{definition}

\begin{definition}
A \emph{triangle} in a suspended category is a diagram of the form
$$X\to Y\to Z\to X[1].$$
The morphism $X\to Y$ is called a \emph{base} of the triangle.
\end{definition}

\begin{definition}
An \emph{octahedron} in a suspended category is a commutative diagram of the form
\begin{equation}
\label{eq_oct}
\xymatrix{
X\ar[r]^f \ar@{=}[d] & Y\ar[r]^{x}\ar[d]^g & Z'\ar[r]\ar[d] & X[1]\ar@{=}[d]\\
X\ar[r] & Z\ar[r]\ar[d] & 	Y'\ar[r] \ar[d] & X[1]\ar[d]^{f[1]}\\
& X'\ar@{=}[r]\ar[d]^r & X'\ar[r]^r\ar[d] &Y[1]\\
& Y[1]\ar[r]^{x[1]} & Z'[1].& 
}
\end{equation}
A pair of morphisms $X\xra{f} Y\xra{g} Z$ is called a \emph{base} of the octahedron.
\end{definition}

\begin{definition}
\label{def_trcat}
A \emph{triangulated} category is a suspended category $\CC$ with a class of triangles  called \emph{distinguished triangles} satisfying the following axioms:
\begin{enumerate}
\item[(1a)] Any triangle isomorphic to a distinguished one is distinguished.
\item[(1b)] For any object $X$ in $\CC$ the triangle $0\to X\xra{1_X} X\to 0$ is distinguished.
\item[(1c)] A triangle $X\xra{f} Y\xra{g} Z\xra{h} X[1]$ is distinguished if and only if 
the triangle $Y\xra{g} Z\xra{h} X[1]\xra{-f[1]} Y[1]$ is distinguished.
\item[(2)] Any morphism $f\colon X\to Y$ fits into some distinguished triangle
$X\xra{f} Y\to Z\to X[1]$.
\item[(3)] For any two distinguished triangles, any morphism of their bases extends to a morphism of triangles.  That is, for any morphisms $u,v$ making the left square in  
$$
\xymatrix{
X\ar[r]^f\ar[d]^u & Y\ar[r]^g\ar[d]^v& Z\ar[r]^h\ar@{-->}[d]^w & X[1]\ar[d]^{u[1]}\\
X'\ar[r]^{f'}& Y'\ar[r]^{g'} & Z\ar[r]^{h'} & X[1]
}
$$
commutative there exists a morphism $w$ completing the diagram.
\item[(4)] Any pair of morphisms $X\xra{f} Y\xra{g} Z$ fits into an octahedron (\ref{eq_oct}) 
whose first two rows and two central columns are distinguished triangles. Such octahedra are called \emph{distinguished}.
\end{enumerate}
\end{definition}

\begin{definition}
A functor  between triangulated categories (or more generally between suspended categories with a class of distinguished triangles) is called \emph{exact} if
it is stable and preserves distinguished triangles.
\end{definition}

\begin{definition}
A monad $\SS=(S,\eta,\mu)$ on a triangulated category (or more generally on a suspended category with a class of distinguished triangles) is called \emph{exact} if
$S$ is an exact functor and $\eta,\mu$ are stable natural transformations.
\end{definition}

\medskip
We refer to \cite{BLL}, \cite{Ke} or \cite{Ke2} for the definitions and basic facts concerning DG-categories.

Let $\k$ be a field. A \emph{DG-category} is a $\k$-linear category such that all $\Hom$ spaces are differential complexes of $\k$-vector spaces and composition of morphisms satisfies graded Leibniz rule.
For a DG-category $\AA$, by $H^0(\AA)$ the \emph{homotopic} category of $\AA$ is denoted. This is a category with the same objects as $\AA$ and whose $\Hom$ spaces are zero homology of $\Hom$ spaces in $\AA$. An additive functor $\Phi\colon \AA\to\BB$ between two DG-categories is a \emph{DG-functor} if for any $X,Y\in \AA$ the induced morphism $\Hom_{\AA}(X,Y)\to \Hom_{\BB}(\Phi(X),\Phi(Y))$ is a morphism of complexes. For any DG-functor $\Phi\colon \AA\to\BB$ one has an induced functor on homotopy categories $H^0(\Phi)\colon H^0(\AA)\to H^0(\BB)$. A DG-functor  $\Phi$ is said to be \emph{quasi-equivalence} if for any $X,Y\in \AA$ the morphism $\Hom_{\AA}(X,Y)\to \Hom_{\BB}(\Phi(X),\Phi(Y))$ is a quasi-isomorphism of complexes and $H^0(\Phi)$ is essentially surjective.

Let $C_{DG}(\k)$ denote the DG-category of complexes of $\k$-vector spaces. For a DG-category $\AA$, a (right) \emph{$\AA$-module} is a DG-functor from $\AA^{op}$ to $C_{DG}(\k)$. Denote by $\Mod{-}\AA$ the category of right $\AA$-modules, it is a DG-category. It has shift and cones of closed morphisms of degree zero, its homotopy category $H^0(\Mod{-}\AA)$ is triangulated. 
Yoneda embedding $h\colon \AA\to \Mod{-}\AA$ is a fully faithful functor.  Modules of the form $h^X, X\in\AA$, are called \emph{free}. The minimal full strict subcategory in $\Mod{-}\AA$, containing all free modules and closed under shifts and cones is called \emph{pretriangulated hull} of $\AA$, we denote it $\Pretr(\AA)$. Its homotopy category $H^0(\AA)$ is triangulated. An $\AA$-module 
$M$ is said to be \emph{semi-free} if there exists a filtration $$0=M_0\subset M_1\subset \ldots \subset M$$ of submodules such that $\cup M_i=M$ and $M_i/M_{i-1}$ is isomorphic to a direct sum of shifts of some free modules. An $\AA$-module is \emph{perfect} if it is semi-free and isomorphic in $H^0(\Mod{-}\AA)$ to a direct summand of some module in $\Pretr(\AA)$. DG-category of perfect $\AA$-modules is denoted $\Perf(\AA)$. Its homotopy category $H^0(\Perf(\AA))$ is also triangulated and it is an idempotent closure of $H^0(\Pretr(\AA))$.

If Yoneda embedding $h\colon \AA\to\Pretr(\AA)$ is a quasi-equivalence, then 
$\AA$ is called \emph{pretriangulated}. If $h\colon \AA\to\Pretr(\AA)$ is a DG-equivalence, then $\AA$ is called \emph{strongly pretriangulated}. DG-category $\AA$ is said to be \emph{perfect} if the embedding $\AA\to\Perf{\AA}$ is a quasi-equivalence.
In all three cases the homotopy category $H^0(\AA)$ is triangulated.

\section{Triangulated structure on the quotient category of a triangulated category}
\label{section_balmer}

Let $\TTT$ be a triangulated category with an action of a group $G$ by exact autoequivalences~$\phi_g, g\in G$. The key subject of this paper is triangulated structure on the quotient category~$\TTT^G$. 

\begin{definition}
\label{def_TTGtriang}
Define a shift functor on $\TTT^G$: on objects $(F,(\th_g))[1]=(F[1],(\th_g[1]))$, on morphisms in $\TTT^G$ shift is the same as on morphisms in $\TTT$.
Say that a triangle $$(F_1,(\th^1_g))\xra{\a} (F,(\th_g))\xra{\b} (F_2,(\th^2_g))\xra{\gamma} (F_1,(\th^1_g))[1]$$ in $\TTT^G$
is distinguished if and only if  the triangle
$$F_1\xra{\a} F\xra{\b} F_2\xra{\gamma} F_1[1]$$ is distinguished in $\TTT$.
\end{definition}

Essentially, this definition introduces a triangulated structure on $\TTT^G$.
It follows from results of P.\,Balmer \cite{Ba2}, an overview of which is given below.

\medskip
Let $\TTT$ be a suspended category with a  class of distinguished triangles and $\SS$ an exact monad on $\TTT$. Then one can introduce a shift functor  and a class of distinguished triangles in $\TTT^{\SS}$ like in Definition \ref{def_TTGtriang}:
shift is defined in an obvious way, a triangle in $\TTT^{\SS}$ is called distinguished if and only if its image in $\TTT$ under forgetful functor is distinguished.

To formulate Balmer's results one needs to use a modification of standard Verdier's axioms of a triangulated category, which was proposed by M.\,K\"unzer \cite{Ku}. 
See also G.\,Maltsiniotis~\cite{Ma}.
For any $n\ge 2$ they define
triangulated categories of order~$n$ by introducing distinguished $n$-triangles. We give the definitions only for $n=2$ and $3$. 
\begin{definition}[\cite{Ba2}, Definition 5.11]
It is said that a suspended category $\TTT$ with a class of distinguished triangles \emph{has triangulation of order $2$} if it satisfies axioms (1)--(3) from Definition~\ref{def_trcat}. It is said that  $\TTT$  \emph{has triangulation of order $3$} if $\TTT$ satisfies all axioms of Definition~\ref{def_trcat} plus axiom (5): for any two distinguished octahedra any morphism between their bases extends to a morphism of octahedra.
\end{definition}
\begin{remark}
Note that $2$-triangles are ordinary triangles and $3$-triangles are octahedra.
\end{remark}
\begin{remark}
\label{remark_2t3}
Triangulated of order $3$ $\Longrightarrow$ Triangulated $\Longrightarrow$ Triangulated of order $2$. 
\end{remark}

In fact, all triangulated categories that appear in algebra or geometry are triangulated of any order:
\begin{predl}
\label{prop_model}
A homotopy category of any stable model category is triangulated of any order. In particular, a triangulated category which has a DG-enhancement is triangulated of any order.
\end{predl}
\begin{proof}
The first statement is~\cite[Remark 5.12]{Ba2}. For the second, suppose $\AA$ is a pretriangulated DG-category. Then the DG-category $\Mod{-}\AA$  of right DG-modules over $\AA$
admits a stable model structure, see~\cite[Theorem 3.2]{Ke2}. The homotopy category of the model category $\Mod{-}\AA$ is the derived category $\D(\AA)$ of $\AA$, therefore $\D(\AA)$ is triangulated of any order.
Homotopy category $H^0(\AA)$ is a fully faithful subcategory in $\D(\AA)$. Since $\AA$ is pretriangulated, $H^0(\AA)$ is closed under all degrees of shift functor and forming of (ordinary) triangles. It follows that $H^0(\AA)$ is closed under forming of higher triangles, hence $H^0(\AA)$ is also triangulated of any order. 
\end{proof}

Also we recall 
\begin{definition}
A functor $F\colon \CC\to\DD$ between two categories is called \emph{separable} if the natural transformation of functors $$F\colon \Hom_{\CC}(-,-)\to \Hom_{\DD}(F-,F-)$$
from $\CC^{op}\times \CC$ to $\mathcal{S}ets$ has a left inverse natural transformation. Suppose also that $\CC$ and $\DD$ are suspended categories and $F$ is a stable functor. Then $F$ is called \emph{stably separable} if the above natural transformation of functors has a left inverse stable natural transformation.
\end{definition}
In practice, separability of a functor can often be checked using
\begin{lemma}[See {\cite[Remark 3.9]{Ba2}} or Rafael {\cite[Theorem 1.2]{Ra}}]
\label{lemma_sep}
In the above assumptions suppose that $F$ has a left adjoint functor $H$. Then $F$ is separable if the counit morphism $HF\to \Id_{\CC}$ has a right inverse morphism of functors $\Id_{\CC}\to HF$. In the suspended situation, $F$ is stably separable if the counit morphism of functors has a right inverse stable morphism.
\end{lemma}

Now we can formulate
\begin{theorem}[{\cite[Theorem 5.17]{Ba2}}]
\label{th_balmer}
Let $\TTT$ be a suspended idempotent complete category with a triangulation of order $2$ or $3$.
Let $\SS$ be an exact monad on $\TTT$. Suppose the forgetful functor $\TTT^{\SS}\to \TTT$ is stably separable. 
Then the category $\TTT^{\SS}$ has a triangulation of order $2$ or $3$ respectively such that a triangle in $\TTT^{\SS}$ is distinguished if and only if its image in $\TTT$ is distinguished.
\end{theorem}

From this we deduce the following 
\begin{theorem}
\label{th_TTTGtriang}
Let $\TTT$ be a suspended idempotent complete category linear over a ring~$\k$ with a triangulation of order $2$ or $3$. Suppose a finite group $G$ acts on $\TTT$ by exact autoequivalences and $|G|$ is invertible in $\k$. Then Definition~\ref{def_TTGtriang} makes $\TTT^G$ a triangulated category of order $2$ or $3$ respectively. 
\end{theorem}
\begin{proof}
Consider adjoint functors $p^*\colon \TTT^G\to\TTT$ and $p_*\colon \TTT\to\TTT^G$ defined in Section~\ref{section_group}. Let $\SS=\SS(p_*,p^*)$ be the monad on $\TTT$ defined by the adjoint pair $p_*,p^*$. By Proposition~\ref{prop_four}, we have an equivalence $\TTT^G\to\TTT^{\SS}$ which makes a commutative diagram 
$$\xymatrix{\TTT^G\ar[rr]^{\sim}\ar[rd]_{p^*}&& \TTT^{\SS}\ar[ld]^{Q^*}\\
&\TTT.&\\}
$$ 
We define shift functors and distinguished triangles in $\TTT^G$ and $\TTT^{\SS}$ as explained above. Clearly, equivalence $\TTT^G\to\TTT^{\SS}$ is exact. Therefore it is enough to check that $\TTT^{\SS}$ is triangulated (of order $2$ or $3$).
Since adjunction morphisms for $p_*$ and $p^*$ are stable, the monad $\SS$ is exact. 
We apply Theorem~\ref{th_balmer}. It suffices to check that the functor $Q^*\colon \TTT^{\SS}\to \TTT$ is stably separable, or equivalently that $p^*\colon \TTT^G\to\TTT$ is stably separable.
By Lemma~\ref{lemma_sep}, we need to check that the counit morphism of functors $\epsilon'\colon p_*p^*\to\Id_{\TTT^G}$ has a right inverse.  That is true: by Lemma~\ref{lemma_1G} the right inverse morphism is given by 
$$\frac1{|G|}\eta\colon \Id_{\TTT^G}\to p_*p^*.$$ 
Explicit construction of $\eta$ (see Proof of Lemma~\ref{lemma_biadjunction}) implies that the above morphism is stable, so Theorem~\ref{th_balmer} applies.
\end{proof}


\begin{corollary}
\label{corollary_TTTGtriang}
Let $\TTT$ be a triangulated category linear over a ring $\k$. Let a finite group~$G$ act on $\TTT$ by exact autoequivalences, assume $|G|$ is invertible in $\k$. 
Suppose $\TTT$ has a DG-enhancement~$\AA$. Then Definition~\ref{def_TTGtriang} makes $\TTT^G$ a triangulated category. 
\end{corollary}
\begin{proof}
First, suppose $\TTT$ is idempotent complete.
By Proposition~\ref{prop_model}, category $\TTT$ is triangulated of order $3$. By Theorem~\ref{th_TTTGtriang}, category $\TTT^G$ is also triangulated of order $3$, hence triangulated.

The general case is reduced to the case considered above by passing to idempotent completion. Let $\bar\TTT$ be the idempotent completion   of $\TTT$. We extend $G$-action from $\TTT$ to $\bar\TTT$ as explained in Proposition~\ref{predl_extend2}. Note that the action on $\bar\TTT$ is exact.
Also note that DG-category $\Perf(\AA)$ is a DG-enhancement for $\bar\TTT$, see beginning of Section~\ref{section_dg} for details. Therefore, by the above arguments, category $\bar\TTT^G$ is triangulated with triangles as in Definition~\ref{def_TTGtriang}. One has a commutative diagram of categories and exact functors
$$\xymatrix{\TTT^G\ar[rr]^{p^*}\ar@{^{(}->}[d] && \TTT\ar@{^{(}->}[d]\\
\bar\TTT^G\ar[rr]^{p^*}&& \bar\TTT,
}$$
where $\TTT^G$ is a full subcategory in $\bar\TTT^G$ consisting of objects that are mapped by $p^*$ to objects of $\TTT$. It follows that $\TTT^G\subset \bar\TTT^G$ is a triangulated subcategory.
\end{proof}

\section{Finite group quotients for derived categories}
\label{section_derived}

Suppose $\TTT$ is the bounded derived category of an abelian category $\AA$.
Consider an action of a group $G$ on $\TTT$ induced by a $G$-action on $\AA$. For an abelian category $\AA$ the category $\AA^G$ is also abelian. In this section we demonstrate directly that the category $\TTT^G$ is triangulated by proving that $\TTT^G\cong \D^b(\AA^G)$.

\begin{theorem}
\label{th_AGAG}
Let $\AA$ be an abelian category with an action of a finite group $G$. Suppose~$\AA$ is linear over a ring $\k$ and $|G|$ is invertible in $\k$. Let $\D^b(\AA)$ be its bounded derived category, 
it is equipped
with an action of $G$ in the natural way.  Then one has an equivalence $\D^b(\AA^G)\to \D^b(\AA)^G$.
\end{theorem}
\begin{proof}
First of all, we note that $\D^b(\AA^G)$ is idempotent complete by~\cite[Corollary 2.10]{Ba}.

Consider the functors $p^*\colon \AA^G\to\AA$ and $p_*\colon \AA\to \AA^G$ introduced in Section~\ref{section_group}. Since they are exact, there exist derived functors $Rp^*\colon \D^b(\AA^G)\to\D^b(\AA)$ and $Rp_*\colon \D^b(\AA)\to \D^b(\AA^G)$ which can be defined termwise. 
Also consider the adjoint functors 
$q^*\colon \D^b(\AA)^G\to \D^b(\AA)$ and $q_*\colon \D^b(\AA)\to \D^b(\AA)^G$, see Section~\ref{section_group}.

Adjoint pairs $Rp^*,Rp_*$ and $q^*,q_*$ define two comonads on $\D^b(\AA)$, which are tautologically isomorphic.

Use Proposition~\ref{prop_suffcond} to check that the comparison functor
$$\D^b(\AA^G)\to \D^b(\AA)_{\TT(Rp^*,Rp_*)}$$
is an equivalence. We need to check that the 
canonical morphism of functors $\Id\to Rp_*Rp^*$ on $\D^b(\AA^G)$ is a split embedding.
Indeed, for any $$\FF^{\bul}=[\ldots\to (F^i,(\th^i_g))\to (F^{i+1},(\th^{i+1}_g))\to \ldots]\in\D^b(\AA^G)$$
 the morphism of complexes
$$\FF^{\bul}\to Rp_*Rp^*\FF^{\bul}$$
given by the family 
$$\oplus_h \th^i_h\colon (F^i,(\th^i_g))\to (\oplus_{h\in G} \phi_h(F^i),(\xi^i_g)),$$ 
has a left inverse morphism
$$Rp_*Rp^*\FF^{\bul}\to \FF^{\bul}$$
given by the family 
$$\frac1{|G|}\oplus_h (\th^i_h)^{-1}\colon (\oplus_{h\in G} \phi_h(F^i),(\xi^i_g))\to (F^i,(\th^i_g)).$$
Clearly, this splitting is functorial.

We obtain a series of equivalences
$$\D^b(\AA^G)\to \D^b(\AA)_{\TT(Rp^*,Rp_*)}=\D^b(\AA)_{\TT(q^*,q_*)}\cong  \D^b(\AA)^G$$
where the latter equivalence is due to Proposition~\ref{prop_four}.
\end{proof}
\begin{remark}
The only reason why we need the derived category to be bounded is to check its idempotent completeness by using~\cite[Corollary 2.10]{Ba} of Balmer and Schlichting. Theorem~\ref{th_AGAG} also holds for unbounded, left or right-bounded derived category of $\AA$ provided that this derived category  is idempotent complete.
\end{remark}

As corollaries, we obtain theorems from the introduction.

\begin{theorem}
\label{th_descentforequiv}
Let $G$ be a finite group and $X$ be a quasi-projective $G$-variety over a field~$\k$. 
Suppose $\har(\k)$ does not divide $|G|$. Then
$$\D^b(\coh(X))^G\cong \D^b(\coh^G(X)).$$
Informally, ``passing to equivariant category commutes with passing to derived category''.
\end{theorem}
\begin{proof}
Take $\AA=\coh(X)$. Then $\AA^G\cong \coh^G(X)$.
By Theorem~\ref{th_AGAG}, we get the result.
\end{proof}

\begin{corollary}
\label{cor_galois}
Suppose $X$ is a Galois covering of a quasi-projective variety $Y$ over a field $\k$ with a Galois group $G$. 
Suppose $\har(\k)$ does not divide $|G|$. Then
$$\D^b(\coh(X))^G\cong \D^b(\coh(Y)).$$
\end{corollary}
\begin{proof}
It follows from Theorem~\ref{th_descentforequiv} and the well-known fact that $\coh^G(X)\cong\coh(Y)$. 
\end{proof}

\begin{theorem}
\label{th_descentforcoverings}
Let $X$ be a quasi-projective algebraic variety over a field $\k$ and $G\subset \Pic(X)$ be a finite subgroup. Let $G$ act on $\coh(X)$ by tensoring into line bundles of $G$.
Let $$Y=\Spec_X\left(\bigoplus_{\LL\in G}\LL^{-1}\right)$$
be the relative spectrum.
Suppose $\har(\k)$ does not divide $|G|$. Then
$$\D^b(\coh(X))^G\cong \D^b(\coh(Y)).$$
\end{theorem}
\begin{proof}
Since $\Pic(X)$ is a not a set of line bundles, but a set of isomorphism classes of line bundles, certain care should be taken when defining $G$-action on $\coh(X)$. Let us do it in some details.

Clearly, 
$G\cong \langle g_1\rangle\times\ldots\times \langle g_m\rangle$ where $g_i\in\Pic(X)$ are elements of order $n_i$. Choose a line bundle $\LL_i$ on $X$ representing each $g_i$.
Fix isomorphisms $t_i\colon \LL^{n_i}\to \O_X$ for each~$i$. For $g=\prod g_i^{d_i}$, $0\le d_i<n_i$ denote by $\LL(g)$ the bundle $\otimes_i \LL_i^{d_i}$. 
Define an action of $G$ on $\coh(X)$. Let $\phi_g\colon \coh(X)\to\coh(X)$
be $\LL(g)\otimes -$. Isomorphisms $\e_{g,h}\colon \phi_g\phi_h\to\phi_{hg}$ are defined through isomorphisms $\LL(g)\otimes \LL(h)\cong \LL(hg)$ which are tautological or defined via~$t_i$.

Let 
$$\RR=\bigoplus_{g\in G} \LL(g)^{-1}$$
be a sheaf on $X$. With the use of $t_i$, one can introduce multiplication $\RR\otimes_{\O_X}\RR\to\RR$  making $\RR$ a sheaf of $\O_X$-algebras.
Let $Y=\Spec_X\RR$ be the relative spectrum of $\RR$. Coherent sheaves on $Y$ are coherent sheaves of $\RR$-modules on $X$. A coherent sheaf $\FF$ on $X$ is a sheaf of $\RR$-modules if it is equipped with a morphism $a\colon \RR\otimes \FF\to \FF$ compatible with multiplication.
One has 
\begin{multline*}
a\in\Hom(\RR\otimes\FF,\FF)=\prod_{g\in G}\Hom(\LL(g)^{-1}\otimes\FF,\FF)=\\
=\prod_{g\in G}\Hom(\FF,\LL(g)\otimes\FF)=\prod_{g\in G}\Hom(\FF,\phi_g(\FF))\ni (\theta_g).
\end{multline*}
It can be checked that $a$ is compatible with multiplication in $\RR$ iff $(\theta_g)$ is compatible with $\e_{g,h}$ in the sense of Definition~\ref{def_C^G}.
Thus
coherent sheaves of $\RR$-modules correspond to $G$-equivariant coherent sheaves on $X$
with respect to the action introduced above. Therefore 
$$\coh(X)^G\cong \coh(Y).$$
Let $\AA=\coh(X)$, then $\AA^G\cong \coh(Y)$.
By Theorem~\ref{th_AGAG}, we have $$\D^b(\coh(X))^G\cong \D^b(\coh(Y)).$$
\end{proof}

\section{Finite group quotients for enhanced triangulated categories}
\label{section_dg}
Suppose a triangulated category $\TTT$ with a $G$-action has a DG-enhancement:
a pretriangulated DG-category $\AA$ with an equivalence $H^0(\AA)\to \TTT$. 
In this section we address the following question: does equivariant category $\TTT^G$ have any reasonable DG-enhancement? We give the  positive answer assuming that the $G$-action on $\TTT$ is induced by a DG-action of~$G$ on $\AA$. Under that hypothesis we construct a pretriangulated DG-category $Q_G(\AA)$ and an exact equivalence $H^0(Q_G(\AA))\to \TTT^G$. This construction has good functorial properties.

Definitions of a group action and of an equivariant object are to be modified in the case of DG-categories, they are as follows:

\begin{definition}
A \emph{(right) action} of a group $G$ on a DG-category $\AA$ consists of the following data:
\begin{itemize}
\item family of DG-autoequivalences $\phi_g\colon \CC\to \CC, g\in G$;
\item family of closed isomorphisms of degree $0$: $\e_{g,h}\colon \phi_g\phi_h\to \phi_{hg}$, satisfying usual associativity conditions.
\end{itemize}
\end{definition}

\begin{definition}
A \emph{$G$-equivariant} object in a DG-category $\AA$ is a pair $(F,(\th_g)_{g\in G})$ where $F\in \Ob\AA$ and $(\th_g)_{g\in G}$ is a family of closed isomorphisms of degree $0$
$$\th_g\colon F\to \phi_g(F),$$
satisfying usual associativity conditions.
A \emph{morphism} of $G$-equivariant objects from $(F_1,(\th^1_g))$ to $(F_2,(\th^2_g))$ is a morphism $f\colon F_1\ra F_2$ compatible with $\theta_g$.
\end{definition}

\begin{predl}
For an action of a group $G$ on a DG-category $\AA$, the category of equivariant objects 
$\AA^G$ is also a DG-category.
\end{predl}
\begin{proof}
Indeed, it is clear that 
$$\Hom_{\AA^G}((F_1,(\th^1_g)),(F_2,(\th^2_g)))\subset \Hom_{\AA}(F_1,F_2)$$
is a subcomplex.
\end{proof}

For a pretriangulated DG-category $\AA$ with a $G$-action the category $\AA^G$ may not be pretriangulated, see an example below. But for a strongly pretriangulated category~$\AA$, the category $\AA^G$ is also strongly pretriangulated, see~\cite[Prop. 3.7]{So}. 

\begin{example}
\label{ex_notpretr}
We give an example of a pretrianguleted DG-category $\AA_0$ with a $G$-action such that the category $\AA_0^G$ is not pretriangulated.

Let $\CC$ be the category of $\Z/3\Z$-graded vector spaces.
Denote by $V_i, (i=0,1,2)$ the simple objects of $\CC$. 
Let $\AA=C^{\bul}_{DG}(\CC)$ be the DG-category of complexes over $\CC$.
Let $$M_i=V_0\oplus [V_i\xra{\Id} V_i]$$ be the complex located in degrees $-1$ and $0$.

Let $\AA'_0\subset \AA$ be the full subcategory whose objects are all objects of $\AA$ except for those quasi-isomorphic to $V_0$, let $\AA_0\subset \AA$ be the full subcategory such that $\Ob \AA_0=\Ob \AA_0'\cup \{M_1,M_2\}$. Since $M_1\cong V_0$ in $H^0(\AA)$, the category $\AA_0$ is pretriangulated (but not strongly pretriangulated).
Let the group $G=\Z/2\Z=\langle g\rangle$ act on $\AA$ by permuting $V_1$ and $V_2$ and sending $V_0$ to itself. Then the subcategory $\AA_0$ is invariant.

We claim that the category $\AA_0^G$ is not pretriangulated.
Indeed, there is an object $(V_0, (1)_g)[-1]$ in $\AA_0^G$. But $\AA_0^G$ contain no objects $(F,(\theta_g))$ quasi-isomorphic to $(V_0, (1)_g)$. Assume the contrary. Then $F$ is quasi-isomorphic to $V_0$. The definition of $\AA_0$ implies that $F$ is either $M_1$ or $M_2$. 
In both cases $F$ is not DG-isomorphic to $\phi_{g}(F)$, we get the contradiction.
Therefore $\AA_0^G$ is not homotopically closed under shifts and hence is not triangulated.
\end{example}

Suppose a triangulated category $\TTT$ has an enhancement: a pretriangulated DG-category~$\AA$ and an exact equivalence $H^0(\AA)\to \TTT$. Suppose the finite group $G$ acts on both
$\AA$ and $\TTT$ compatibly. Then P. Sosna in~\cite{So} defines $\TTT^G$ as $H^0(\Pretr(\AA^G))$. Below we demonstrate that this construction is, in general, dependent on the choice of enhancement.

\begin{example}
We give an example of two strongly pretriangulated categories $\AA_1$ and $\AA_2$ with actions of a finite $G$ and of $G$-equivariant quasi-equivalence $\AA_1\to \AA_2$ 
such that the induced functor 
$$ \Pretr(\AA_1^G)\to \Pretr(\AA_2^G)$$
is not a quasi-equivalence.

Let $\CC$, $\AA$, $V_i$ and $M_i$ be as in Example~\ref{ex_notpretr}. Denote by  $(\dim_0,\dim_1,\dim_2)\in \Z^3$
the dimension of objects in $\CC$.

Consider the subcategory $\AA_1\subset \AA$ generated by $M_1$ and $M_2$ by taking shifts and cones. Consider the subcategory $\AA_2\subset \AA$ generated by $M_1, M_2$ and $V_0$. 
Clearly, $\AA_i$ are strongly pretriangulated, the inclusion $\AA_1\to\AA_2$ induces an equivalence $H^0(\AA_1)\cong H^0(\AA_2)\cong \D^b(\mathrm{vect})$ with $V_0\cong M_1\cong M_2$ in $H^0(\AA_i)$ being the simple object. Hence, $\AA_1$ and $\AA_2$ are quasi-equivalent.

Let the group $G=\Z/2\Z$ act on $\AA$ by permuting $V_1$ and $V_2$ and sending $V_0$ to itself. Then the subcategories $\AA_1$ and $\AA_2$ are invariant. Since they are strongly pretriangulated, the categories $\AA_i^G$ are pretriangulated and $\Pretr(\AA_i^G)$ is DG-equivalent to $\AA_i^G$. Clearly,  $H^0(\AA_2^G)\cong \D^b(\mathrm{vect}^G)$, its simple objects are $(V_0,(1)_g)$
and $(V_0,(\mathrm{sign}(g))_g)$. We claim that the subcategory 
$H^0(\AA_1^G)\subset H^0(\AA_2^G)$ does not contain objects isomorphic to  $(V_0,(1)_g)$, and hence the inclusion $H^0(\AA_1^G)\to H^0(\AA_2^G)$ is not an equivalence.

Indeed, let $(N^{\bul},(\th_g))$ be an object of $\AA_1^G$. 
Note that 
$$\dim_1(N^i)+\dim_2(N^i)=\dim_0(N^i)+\dim_0(N^{i+1}).$$
(This is true for $N^{\bul}$ being any shift of $M_1$ and $M_2$ and therefore for any complex obtained from them by subsequent taking cones.)
Since $N^{\bul}$
is $G$-invariant, one has $\dim_1(N^i)=\dim_2(N^i)$.
We deduce that $$\dim_0(N^i)=\dim_0(N^{i+1})\pmod{2}.$$
Since the complex $N^{\bul}$ is finite, all $\dim_0(N^i)=0\pmod{2}$.
Therefore $$\sum_i(-1)^i\dim_0H^i(N^{\bul})=0\pmod{2}.$$ Hence $N^{\bul}$ is not homotopic to $V_0$.
\end{example}

This issue arises because the enhancement $\AA$ may be ``not enough symmetric'': objects $F$ and $\phi_g(F)$ that should be DG-isomorphic are only homotopic. Therefore
the category $H^0(\AA^G)$ lacks some desired objects. Fortunately, these missing objects can be recovered as certain direct summands of objects of $H^0(\AA^G)$. More precisely, if $\TTT$ is idempotent complete, then the idempotent completion of $H^0(\AA^G)$ is the good candidate for $\TTT^G$: it does not depend on the enhancement.

\begin{lemma}
\label{lemma_main}
Let $\AA$ be an additive DG-category, linear over a ring $\k$. Suppose a finite group $G$ acts on $\AA$ and $|G|$ is invertible in $\k$. Then one has a natural equivalence
$$H^0(\Perf(\AA^G))\to H^0(\Perf(\AA))^G.$$
\end{lemma}
\begin{proof}
Consider the DG-category $\AA^G$ and adjoint functors $p^*\colon \AA^G\to \AA$ and $p_*\colon \AA\to \AA^G$  introduced in Section~\ref{section_group}. 
They are both left and right adjoint to each other and we have natural transformations of adjunction: $\eta\colon \Id_{\AA^G}\to p_*p^*$ and $\e'\colon p_*p^*\to \Id_{\AA^G}$, such that $\e'\eta=|G|$. 
These functors extend to adjoint functors  $\Perf(\AA^G)\to \Perf(\AA)$ and $\Perf(\AA)\to \Perf(\AA^G)$, which we also denote $p^*$ and $p_*$ respectively. They also possess   
natural transformations $\eta\colon \Id_{\Perf(\AA^G)}\to p_*p^*$ and $\e'\colon p_*p^*\to \Id_{\Perf(\AA^G)}$  satisfying the same identity. The same is true for $H^0$:
we have got adjoint functors 
$$H^0(p^*)\colon H^0(\Perf(\AA^G))\to H^0(\Perf(\AA))\quad\text{and}\quad
H^0(p_*)\colon H^0(\Perf(\AA))\to H^0(\Perf(\AA^G))$$ and functorial morphisms
$$\eta\colon \Id_{H^0(\Perf(\AA^G))}\to H^0(p_*)H^0(p^*)\qquad\text{and}\qquad\e'\colon H^0(p_*)H^0(p^*)\to \Id_{H^0(\Perf(\AA^G)}$$
such that $\e'\eta=|G|$.
This adjoint pair generates a comonad $\TT=\TT(H^0(p^*),H^0(p_*))$ on $H^0(\Perf(\AA))$.
By Proposition~\ref{prop_suffcond}, we have an equivalence
$$\Phi\colon H^0(\Perf(\AA^G))\to H^0(\Perf(\AA))_{\TT}.$$
Indeed, the natural morphism $\eta\colon \Id_{H^0(\Perf(\AA^G))}\to H^0(p_*)H^0(p^*)$
has a left inverse morphism $\frac{\e'}{|G|}$.

Each equivalence $\phi_g\colon \AA\to \AA$ extends to an equivalence $\AA{\mathrm{{-}Mod}}\to \AA{\mathrm{{-}Mod}}$ which restricts to an equivalence $\Perf(\AA)\to \Perf(\AA)$. This defines an action of $G$ on $\Perf(\AA)$.
Clearly, the comonad $\TT$ is isomorphic to the comonad associated with the induced $G$-action on $H^0(\Perf(\AA))$, see Definition~\ref{def_asscomonad}. Therefore by Proposition~\ref{prop_four} one has an equivalence				
$$H^0(\Perf(\AA))^G\to H^0(\Perf(\AA))_{\TT},$$
this concludes the proof.
\end{proof}

As a corollary we get 
\begin{theorem}
\label{th_main}
Let $\TTT$ be an idempotent complete triangulated category with an action of a finite group $G$. Suppose it has a $G$-equivariant enhancement: a pretriangulated DG-category $\AA$ with a $G$-action and a $G$-equivariant exact equivalence $\epsilon\colon H^0(\AA)\to\TTT$.
Then one has an exact equivalence 
$$H^0(\Perf(\AA^G))\to \TTT^G.$$

Consequently, the category $H^0(\Perf(\AA^G))$ depends only on $G$-category $\TTT$ and does not depend on the choice of $\AA$ and of $G$-action on $\AA$. 
\end{theorem}
\begin{proof}
Since $H^0(\AA)\cong\TTT$, the category $H^0(\AA)$ is idempotent complete, so the natural 
embedding $i\colon H^0(\AA)\to H^0(\Perf(\AA))$ is an equivariant equivalence. Using Lemma~\ref{lemma_main}, we get a sequence of equivalences
$$H^0(\Perf(\AA^G))\xra{\Psi} H^0(\Perf(\AA))^G\xra{i^{-1}} H^0(\AA)^G\xra{\epsilon^G} \TTT^G.$$

It remains to check that a triangle in $H^0(\Perf(\AA^G))$ is distinguished iff its image under the above equivalence in $\TTT^G$ is distinguished.

Consider the commutative diagram of functors
$$\xymatrix{
&& H^0(\Perf(\AA))\ar[rr]^{i^{-1}} && H^0(\AA)\ar[rr]^{\epsilon} &&\TTT\\
H^0(\Perf(\AA^G))\ar[rr]^{\Psi}\ar[rru]^{H^0(p^*)}&& H^0(\Perf(\AA))^G\ar[rr]^{i^{-1}} \ar[u]^{p^*}&& H^0(\AA)^G\ar[rr]^{\epsilon^G} \ar[u]^{p^*}&&\TTT^G\ar[u]^{p^*},}$$
where $p^*$ denotes certain forgetful functors.
Let $\Delta$ be a triangle in $H^0(\Perf(\AA^G))$. By the definition of a distinguished triangle in $\TTT^G$, the triangle $\epsilon^Gi^{-1}\Psi(\Delta)$ is distinguished $\Longleftrightarrow$
$p^*\epsilon^Gi^{-1}\Psi(\Delta)$ is distinguished $\Longleftrightarrow$
$\epsilon^Gi^{-1}p^*\Psi(\Delta)$ is distinguished $\Longleftrightarrow$
$p^*\Psi(\Delta)$ is distinguished (because $\epsilon^G$ and $i$ are exact equivalences) 
$\Longleftrightarrow$ $H^0(p^*)(\Delta)$ is distinguished. So we need to demonstrate that 
$\Delta$ is distinguished $\Longleftrightarrow$ $H^0(p^*)(\Delta)$ is distinguished. 
Implication $\Longrightarrow$ is clear. To check the opposite, suppose that $H^0(p^*)(\Delta)$ is distinguished, then $H^0(p_*)H^0(p^*)(\Delta)$ is also distinguished. By the proof of Lemma~\ref{lemma_main}, the morphism $\Id\to H^0(p_*)H^0(p^*)$ is a split embedding of functors. Hence $\Delta$  is distinguished as a direct summand of a distinguished triangle $H^0(p_*)H^0(p^*)(\Delta)$ (see~\cite[Proposition 1.2.3]{Ne}).
\end{proof}

A case when $\TTT$ is not idempotent complete can be reduced to the one considered above.
The idea is straightforward: we extend $\TTT$ to its idempotent completion $\bar\TTT$. On the level of DG-enhancements, this is done by passing from a DG-category $\AA$ to the category of perfect complexes $\Perf(\AA)$. Then we restrict equivalence $H^0(\Perf(\AA^G))\to H^0(\Perf(\AA))^G\to \bar\TTT^G$ to certain smaller subcategories.
Below we do it in some details.

\medskip
Suppose $\AA$ is a pretriangulated DG-category with a $G$-action.
Let $\Psi\colon H^0(\Perf(\AA^G))\to H^0(\Perf(\AA))^G$ be the equivalence from  Lemma~\ref{lemma_main} and $p^*\colon \Perf(\AA^G)\to \Perf(\AA)$ be the forgetful functor. 
\begin{definition}
Denote by $Q_G(\AA)$ the full subcategory of $\Perf(\AA^G)$, whose objects are such $M$ that  $\Psi(M)$ is isomorphic in $H^0(\Perf(\AA))^G$ to an object of $H^0(\AA)^G$. Or, equivalently, such $M$ that $p^*M\in \Perf(\AA)$ is quasi-isomorphic to an object of $\AA$.
\end{definition}

\begin{theorem}
\label{th_maingeneral}
Suppose $\AA$ is a pretriangulated DG-category with a $G$-action. Then

\begin{enumerate}
\item $Q_G(\AA)$ is a strongly pretriangulated DG-category.
\item There exists an equivalence $\Gamma \colon H^0(Q_G(\AA))\to H^0(\AA)^G$.
\item Suppose $\AA_1, \AA_2$ are two pretriangulated DG-categories equipped with a $G$-action. Then for any $G$-equivariant DG-functor $\phi\colon \AA_1\to\AA_2$ 
one has a DG-functor $Q_G(\phi)\colon Q_G(\AA_1)\to Q_G(\AA_2)$ such that the diagram
$$\xymatrix{
H^0(Q_G(\AA_1))\ar[rr]^{\Gamma_1}\ar[d]^{H^0(Q_G(\phi))} && H^0(\AA_1)^G \ar[d]^{H^0(\phi)^G}  \\
H^0(Q_G(\AA_2))\ar[rr]^{\Gamma_2}                       && H^0(\AA_2)^G  
}$$
commutes (up to an isomorphism). Moreover, if $\phi$ is a quasi-equivalence then $Q_G(\phi)$ is also a quasi-equivalence.
\item Functors $Q_G(\phi)Q_G(\psi)$ and $Q_G(\phi\psi)$ are isomorphic for composable $\phi$ and $\psi$.
\end{enumerate}
\end{theorem}

\begin{proof}

\begin{enumerate}
\item To show that $Q_G(\AA)$ is a strongly pretriangulated DG-category, it suffices to check that $Q_G(\AA)$ is closed under shifts and cones in $\Perf(\AA^G)$. This is clear: consider the case of cones. Suppose  $f\colon M\to N$ is a morphism in $\Perf(\AA^G)$ and $K=C(f)$ is its cone in $\Perf(\AA^G)$. Then $p^*(K)$ is a cone of the morphism $p^*(f)\colon p^*(M)\to p^*(N)$ in $\Perf(\AA)$. Since  $p^*(M),p^*(N)\in\Ob \Perf(\AA)$ are quasi-isomorphic to objects of $\AA$ and $\AA$ is pretriangulated, $p^*(K)$ is also quasi-isomorphic to an object of $\AA$. Therefore, $K$ lies in $Q_G(\AA)$.

\item Denote by $\overline{H^0(\AA)^G}$ the closure of $H^0(\AA)^G$ in $H^0(\Perf(\AA))^G$ under isomorphisms. Then one has a commutative 
diagram of functors where vertical arrows denote embeddings of fully faithful subcategories.
$$\xymatrix{
H^0(\Perf(\AA^G))\ar[rr]^{\Psi} &&H^0(\Perf(\AA))^G\\
H^0(Q_G(\AA))\ar@{-->}[rrd]^{\Gamma}\ar@{^{(}->}[u]\ar[rr]^{\Psi}&& \overline{H^0(\AA)^G}\ar@{^{(}->}[u]\\
 &&H^0(\AA)^G.\ar@{^{(}->}[u]^{\s} 
}$$

By the definition of $Q_G(\AA)$, the functor $\Psi\colon H^0(Q_G(\AA))\to \overline{H^0(\AA)^G}$ is a well-defined equivalence. By the definition of $\overline{H^0(\AA)^G}$, the embedding $\s$ is an equivalence. Define $\Gamma$ as a composition of $\Psi$ and an inverse functor to $\s$. Clearly, $\Gamma$ is an equivalence.

\item 
To prove this part, suppose $\phi\colon\AA_1\to\AA_2$ is a DG-functor compatible with $G$-actions. Consider the commutative diagram:
$$\xymatrix{
H^0(\Perf(\AA_1^G))\ar[rrr]^{\Psi_1}\ar[rrdd]^(.3){H^0(\phi^G)} &&&H^0(\Perf(\AA_1))^G\ar[rrdd]^{H^0(\phi)^G} &&\\
H^0(Q_G(\AA_1))\ar@{-}[r]\ar[u]^{} \ar@{-->}[rrdd]^(.4){H^0(Q_G(\phi))} & \ar[rr]^{\Gamma_1} &&H^0(\AA_1)^G\ar@{-}[rd]^{H^0(\phi)^G}\ar[u] &&\\
&&H^0(\Perf(\AA_2^G))\ar[rrr]^{\Psi_2} && \ar[rd] & H^0(\Perf(\AA_2))^G \\
&&H^0(Q_G(\AA_2))\ar[rrr]^{\Gamma_2}\ar[u]^{} &&&H^0(\AA_2)^G.\ar[u] 
}$$
Definition of $Q_G(\AA_i)$ and diagram chase show that the functor $H^0(\phi^G)$ sends objects of subcategory $H^0(Q_G(\AA_1))$ to the objects of $H^0(Q_G(\AA_2))$. Therefore $\phi^G\colon \Perf(\AA_1^G)\to \Perf(\AA_2^G)$ restricts to a functor $Q_G(\phi)\colon Q_G(\AA_1)\to Q_G(\AA_2)$ such that $H^0(Q_G(\phi))$ completes the diagram. 

Finally, if $\phi$ is a quasi-equivalence, then $H^0(\phi)^G\colon H^0(\AA_1)^G\to H^0(\AA_2)^G$ is an equivalence. Since $\Gamma_i$ are equivalences, $H^0(Q_G(\phi))$ is an equivalence.

\item Functors $Q_G(\phi)$, $Q_G(\psi)$ and $Q_G(\phi\psi)$ are restrictions of $\phi,\psi$ and $\phi\psi$ respectively. This implies the result immediately.
\end{enumerate}
\end{proof}

\begin{corollary}
\label{cor_main}
Let $\TTT$ be a triangulated category with an action of a finite group $G$. Suppose it has a $G$-equivariant enhancement: a pretriangulated DG-category $\AA$ with a $G$-action and a $G$-equivariant exact equivalence $\epsilon\colon H^0(\AA)\to\TTT$.
Then there exists an exact equivalence
$H^0(Q_G(\AA))\to \TTT^G$.
\end{corollary}
\begin{proof}
By Theorem~\ref{th_maingeneral}, there is an eqiuvalence
$$H^0(Q_G(\AA))\xra{\Gamma} H^0(\AA)^G\xra{\epsilon^G} \TTT^G.$$
By the proof of Theorem~\ref{th_maingeneral}, $\Gamma$ is a restriction of the equivalence $H^0(\Perf(\AA^G))\to H^0(\Perf(\AA))^G$, which is exact by Theorem~\ref{th_main}. Therefore $\Gamma$ (and $\epsilon^G \Gamma$) is also exact.
\end{proof}

\begin{example}
\label{example_support}
Let $X$ be a quasi-projective variety over $\k$ with an action of a finite group~$G$.
Let $Z\subset X$ be its closed $G$-invariant subvariety. Suppose $\har \k$ does not divide~$|G|$. Denote by $\D^b_Z(\coh(X))$ the full subcategory in $\D^b(\coh(X))$ of objects supported in~$Z$. Clearly, $G$ acts on both $\D^b(\coh(X))$ and $\D^b_Z(\coh(X))$ by pull-back functors. We construct a DG-enhancement for the category $\D^b_Z(\coh(X))^G$.

Denote by $C_{DG}(\O_X\mmod)$ the DG-category of complexes of sheaves of $\O_X$-modules. Let $\mathcal I$ be the full subcategory in $C_{DG}(\O_X\mmod)$ whose objects are left bounded complexes of injective $\O_X$-modules with finite coherent cohomology. It is well-known (see, for example,~\cite[Paragraph 3, Ex.\,3]{BK} or~\cite[Section 5]{BLL}) that $\mathcal I$ is a DG-enhancement of $\D^b(\coh(X)$. Clearly, $G$ acts on $C_{DG}(\O_X\mmod)$ by pullbacks and $\mathcal I$ is an invariant subcategory.
Denote by $\mathcal I_Z$ the full subcategory in $\mathcal I$ of complexes whose cohomology sheaves are supported in~$Z$. Clearly, $\mathcal I_Z$ is a $G$-invariant pretriangulated DG-subcategory in $C_{DG}(\O_X\mmod)$ and $H^0(\mathcal I_Z)\cong \D^b_Z(\coh(X))$. Also, note that $\D^b_Z(\coh(X))$ is idempotent complete. Hence, Theorem~\ref{th_main} can be applied. We obtain that 
$\D^b_Z(\coh(X))^G$ has a DG-enhancement $\Perf(\mathcal I_Z^G)$.
\end{example}


\begin{thebibliography}{99}

\bibitem{Ba} P.\,Balmer, M.\,Schlichting, ``Idempotent completion of triangulated categories'', Journal of Algebra, 236:2 (2001),      819--834.

\bibitem{Ba2} P.\,Balmer, ``Separability and triangulated categories'', 
Adv. Math., 226 (2011), 4352--4372.

\bibitem{TTT} M.\,Barr, C.\,Wells, {\it Toposes, triples and theories}, Reprints in
    Theory and Applications of Categories, No. 12, 2005.

\bibitem{BK} A.\,Bondal, M.\,Kapranov, 
``Enhanced triangulated categories'', Math. USSR Sbornik, 70 (1991), 93Ц-107.

\bibitem{BLL} A.\,Bondal, M.\,Larsen, V.A.\,Lunts, ``Grothendieck ring of pretriangulated categories'', Int. Math. Res. Not. 29 (2004), 1461--1495.


\bibitem{BO} A.\,Bondal, D.\,Orlov, ``Reconstruction of a variety from the derived category and groups of autoequivalences'', Compositio Math., 125:3 (2001), 327--344.

\bibitem{BD} F.\,Borceux, D.\,Dejean, ``Cauchy completion in category theory'', 
Cahiers de Topologie et G\'eom\'etrie Diff\'erentielle Cat\'egoriques, 27:2 (1986),  133--146. 

\bibitem{Ch} X.-W.\,Chen, ``A note on separable functors and monads with an application to equivariant derived categories'', Abhandlungen aus dem Mathematischen Seminar der Universit\"at Hamburg, 85:1 (2015), 43--52.

\bibitem{De} P.\,Deligne, ``Action du groupe des tresses sur une categorie'', Invent. Math.,  128 (1997), {159}Ц-{175}.

\bibitem{El2} A.\,D.\,Elagin, ``Cohomological descent theory for a morphism of stacks and for equivariant derived categories'', Sbornik: Mathematics, 202:4 (2011), 495--526.


\bibitem{KSh} M.\,Kashiwara, P.\,Shapira, {\it Categories and sheaves},
    Springer-Verlag, Berlin, 2006.

\bibitem{Ke} B.\,Keller, ``Deriving DG categories'', 
    Ann. scient. Ec. Norm. Sup., 4 (27) (1994), 63--102.

\bibitem{Ke2} B.\,Keller, ``On differential graded categories'', 	International Congress of Mathematicians, vol. II (2006), Eur. Math. Soc., Z\"urich.

\bibitem{Ku} M.\,K\"unzer, ``On derived categories'', Diploma thesis, Universitat Stuttgart, 1996.

\bibitem{LO} V.\,A.\,Lunts, D.\,O.\,Orlov, ``Uniqueness of enhancement for triangulated categories'', J. AMS, 23 (2010), 853Ц-908.

\bibitem{Lu} J.\,Lurie, {\it Higher Topos Theory}, Princeton University Press, 2009.

\bibitem{ML} S.\,MacLane, {\it Categories for the working mathematician}, Springer Verlag, 1998.

\bibitem{Ma} G.\,Maltsiniotis, ``Categories triangulees superieures'', see preprint at 
\textsl{http://webusers.imj-prg.fr/~georges.maltsiniotis/ps/triansup.ps}, 2005.

\bibitem{Me} B.\,Mesablishvili, ``Monads of effective descent type and comonadicity'', 
Theory and Applications of Categories, 16:1 (2006), 1Ц-45.


\bibitem{Ne} A.\,Neeman, {\it Triangulated categories}, Ann. Math. Studies 148, Princeton UP, 2002.

\bibitem{Or1} D.\,O.\,Orlov, ``On equivalences of derived categories of coherent sheaves on abelian varieties'', Izvestia RAN, Ser.Mat, 66:3 (2002), 131--158.

\bibitem{Ra} M.\,D.\,Rafael, ``Separable functors revisited'', Comm. Algebra, 18:5 (1990), 1445--1459.


\bibitem{So} P.\,Sosna, ``Linearisations of triangulated categories with respect to finite group actions'', arXiv:1108.2144.
    

\end{thebibliography}
\end{document}